\documentclass[12pt]{amsart}
\usepackage{amsmath}
\usepackage{amscd}
\usepackage{amssymb}
\usepackage{amsfonts}
\usepackage{amsthm}
\usepackage{enumerate}
\usepackage{hyperref}
\usepackage{color}
\usepackage{graphicx}
\usepackage{amscd,amsmath,amsthm,amssymb}
\usepackage{mathtools}
\usepackage{subfig}
\usepackage[margin=2cm]{geometry}

\newcommand{\cG}{\mathcal{G}}

\renewcommand{\qedsymbol}{$\square$}

\def\opn#1#2{\def#1{\operatorname{#2}}} 
\opn\chara{char} \opn\length{\ell} \opn\pd{pd} \opn\rk{rk}
\opn\projdim{proj\,dim} \opn\injdim{inj\,dim} \opn\rank{rank}
\opn\depth{depth} \opn\grade{grade} \opn\height{height}
\opn\embdim{emb\,dim} \opn\codim{codim}

\opn\Tr{Tr} \opn\bigrank{big\,rank}
\opn\superheight{superheight}\opn\lcm{lcm}
\opn\trdeg{tr\,deg}
	\opn\reg{reg} \opn\lreg{lreg} \opn\ini{in} \opn\lpd{lpd}
	\opn\size{size} \opn\sdepth{sdepth}
	\opn\link{link}\opn\fdepth{fdepth}\opn\lex{lex}\opn\dist{dist}
	\opn\div{div} \opn\Div{Div} \opn\cl{cl} \opn\Cl{Cl}
	\opn\Spec{Spec} \opn\Supp{Supp} \opn\supp{supp} \opn\Sing{Sing}
	\opn\Ass{Ass} \opn\Min{Min}\opn\Mon{Mon}
	\opn\Ann{Ann} \opn\Rad{Rad} \opn\Soc{Soc}
	\opn\Im{Im} \opn\Ker{Ker} \opn\Coker{Coker} \opn\Am{Am}
	\opn\Hom{Hom} \opn\Tor{Tor} \opn\Ext{Ext} \opn\End{End}
	\opn\Aut{Aut} \opn\id{id}
	
	\opn\nat{nat}
	\opn\pff{pf}
	\opn\Pf{Pf} \opn\GL{GL} \opn\SL{SL} \opn\mod{mod} \opn\ord{ord}
	\opn\Gin{Gin} \opn\Hilb{Hilb}\opn\sort{sort}
	\opn\WHom{WHom}
	\opn\aff{aff} \opn
	\con{conv} \opn\relint{relint} \opn\st{st}
	\opn\lk{lk} \opn\cn{cn} \opn\core{core} \opn\vol{vol}
	\opn\link{link} \opn\star{star}\opn\lex{lex}\opn\set{set}
	\opn\gr{gr}
	
	\def\pot#1#2{#1[\kern-0.28ex[#2]\kern-0.28ex]}

	\opn\dirlim{\underrightarrow{\lim}}
	\opn\inivlim{\underleftarrow{\lim}}
	\let\to=\rightarrow
	
	\def\Implies{\ifmmode\Longrightarrow \else
		\unskip${}\Longrightarrow{}$\ignorespaces\fi}
	\def\implies{\ifmmode\Rightarrow \else
		\unskip${}\Rightarrow{}$\ignorespaces\fi}
	\def\iff{\ifmmode\Longleftrightarrow \else
		\unskip${}\Longleftrightarrow{}$\ignorespaces\fi}

	\let\:=\colon
	\let\epsilon\varepsilon
	\let\kappa=\varkappa
	\def\qed{\ifhmode\textqed\fi
		\ifmmode\ifinner\quad\qedsymbol\else\dispqed\fi\fi}
	\def\textqed{\unskip\nobreak\penalty50
		\hskip2em\hbox{}\nobreak\hfil\qedsymbol
		\parfillskip=0pt \finalhyphendemerits=0}
	\def\dispqed{\rlap{\qquad\qedsymbol}}
	
	\opn\dis{dis}
	\def\pnt{{\raise0.5mm\hbox{\large\bf.}}}
	
	\opn\Lex{Lex}
	
        \newtheorem{Theorem}{Theorem}[section]
	\newtheorem{Lemma}[Theorem]{Lemma}
	\newtheorem{Corollary}[Theorem]{Corollary}
	\newtheorem{Proposition}[Theorem]{Proposition}
	\newtheorem{Remark}[Theorem]{Remark}
	
	\newtheorem{Example}[Theorem]{Example}
	
	\newtheorem{Definition}[Theorem]{Definition}
	\newtheorem{Problem}[Theorem]{Problem}

\begin{document}
   
\title[Ideals of weak graph homomorphisms]{An algebraic study of ideals of weak graph homomorphisms}

\author[F. Navarra]{Francesco Navarra}      
  \address[Francesco Navarra]{Sabanci University, Faculty of Engineering and Natural Sciences, Orta Mahalle, Tuzla 34956, Istanbul, Turkey}	
\email{francesco.navarra@sabanciuniv.edu}

\author[A. A. Qureshi]{Ayesha Asloob Qureshi}      
   \address[Ayesha Asloob Qureshi]{Sabanci University, Faculty of Engineering and Natural Sciences, Orta Mahalle, Tuzla 34956, Istanbul, Turkey}	
\email{aqureshi@sabanciuniv.edu, ayesha.asloob@sabanciuniv.edu}

\author[S. A. Seyed Fakhari]{Seyed Amin Seyed Fakhari}      
   \address[Seyed Amin Seyed Fakhari]{Departamento de Matem\'aticas, Universidad de los Andes, Bogot\'a, Colombia}	
\email{s.seyedfakhari@uniandes.edu.co}



\begin{abstract}
Let $G$ and $H$ be finite simple graphs and assume that either both are undirected or both are directed. We introduce and study the ideal of weak graph homomorphisms $I_{G\to H}$. We characterize all graphs $G$ and $H$ for which every (equivalently, some) power of $I_{G\to H}$ has a linear resolution. Moreover, unmixedness, Cohen-Macaulayness, projective dimension and Castelnuovo-Mumford regularity of these ideals are studied.
\end{abstract}
\maketitle

\section*{Introduction}
The study of monomial ideals and their powers is an active area of research in Combinatorial Commutative Algebra. The general idea is to construct a bridge between algebraic and homological properties of these ideals and a corresponding combinatorial object. Among the most interesting combinatorial  objects are graphs and partially ordered sets (posets). Given any finite poset $P$, in \cite{HH}, Herzog and Hibi defined a squarefree monomial ideal $H_P$ and studied its properties. This ideal is now called the Hibi ideal of $P$. Among all results it is shown in \cite{HH} that any power of $H_P$ has a linear resolution. Moreover, in the same paper, Herzog and Hibi used $H_P$ to  characterize Cohen-Macaulay edge ideals of bipartite graphs. A generalization of Hibi ideals is studied in \cite{EHM}.

Given two finite posets $(P,\leq_P)$ and $(Q, \leq_Q)$, the ideal of poset homomorphisms of $P$ and $Q$ is introduced in \cite{fgh} as follows. Without loss of generality, we may assume that $P$ and $Q$ are posets on $[n]:=\{1,\dots,n\}$ and $[m]:=\{1,\dots,m\}$, respectively. An order preserving map from $P$ to $Q$ is called a {\it poset homomorphism}. In other words, a function $\phi: P\to Q$ is a poset homomorphism if $\phi(x)\leq_Q \phi(y)$ for any pair of elements $x,y\in P$ with $x\leq _P y$. Let $\Hom (P,Q)$ denote the set of all poset homomorphisms from $P$ to $Q$. The {\it ideal of poset homomorphisms} $L(P,Q)$ is a monomial ideal of $S:=K[x_{ij}: i\in [n], j\in [m]]$ ($K$ is a field) which is generated by all the monomials $m_{\phi}:=\prod_{i=1}^nx_{i\phi(i)}$, where $\phi\in \Hom(P,Q)$. This definition extends the notion of generalized Hibi ideals. The ideal of poset homomorphisms is further studied in \cite{AFN, AFN2, HQS, kkm, LW}.

Motivated by the ideal of poset homomorphisms, the goal of this project is to introduce and study a similar ideal associated to a pair of simple graphs. Recall that given two finite simple graphs $G$ and $H$, a map $\phi: V(G)\to V(H)$ is a {\it graph homomorphism} if $\{\phi(x),\phi(y)\}\in E(H)$, for any edge $\{x,y\}\in E(G)$. However, this definition is not really compatible with that of poset homomorphisms. The difference is that for a poset homomorphism $\phi: P\to Q$ and distinct elements $x,y\in P$ with $x\leq_P y$, it is allowed that $\phi(x)=\phi(y)$. However, for a graph homomorphism $\phi$ from $G$ to $H$, if $\{x,y\}\in E(G)$, then by definition, $\phi(x)$ and $\phi(y)$ cannot coincide. To avoid this inconsistency, we relax our setup by considering weak graph homomorphisms (which is also a standard notion in graph theory, \cite[Definition 1.4.3]{K}). The precise definition is as follows. Let $G$ and $H$ be simple graphs on vertex sets $[n]$ and $[m]$, respectively. A
{\it weak homomorphism} from $G$ to $H$ is a map
$\phi:V(G)\to V(H)$ such that for every $\{x,y\}\in E(G)$, either
$\phi(x)=\phi(y)$ or $\{\phi(x),\phi(y)\}\in E(H)$. The set of all weak homomorphisms from $G$ to $H$ will be denoted by $\WHom(G,H)$. Set $S:=K[x_{ij}: i\in[n],\, j\in[m]]$. For each $\phi\in\WHom(G,H)$, define
\[
u_\phi=\prod_{i=1}^{n}x_{i\phi(i)}.
\]
The {\it ideal of weak graph homomorphisms from $G$ to $H$} is
\[
I_{G\to H}=(u_\phi:\phi\in\WHom(G,H)).
\]
The code for computing $I_{G\to H}$ is provided in \cite{N}. 

As the first main result, in Theorem \ref{Theorem:linres}, we characterize all graphs $G$ and $H$ for which every (equivalently, some) power of $I_{G\to H}$ has a linear resolution. It turns out that this is the case if and only if  either $E(G)=\emptyset$ or $H$ is a complete graph. Moreover, the same theorem shows that having a linear resolution and being linearly related are equivalent for ideals of weak graph homomorphisms. Next, we study unmixedness of ideals of weak graph homomorphisms. We prove in Theorem \ref{theorem:unmixed} that $I_{G\to H}$ is unmixed if and only if either $E(G)=\emptyset$ or $H$ is a disjoint union of complete graphs. Using this result, in Theorem \ref{thm:CM+CI}, we characterize all graphs $G$ and $H$ for which $I_{G\to H}$ is Cohen--Macaulay. Also, we prove in Proposition \ref{symbord} that if $I_{G\to H}$ is unmixed, then for each integer $k\geq 1$, the $k$th ordinary and symbolic powers of $I_{G\to H}$ coincide. In the last part of Section \ref{sec1}, we focus on homological invariants of ideals of weak graph homomorphisms and their (symbolic) powers. In Theorem \ref{thm:regundir}, we determine lower bounds for the projective dimension and the Castelnuovo-Mumford regularity of ordinary and symbolic powers of ideals of weak graph homomorphisms. Furthermore, in Corollary \ref{regstar}, we compute the exact value of $\reg I_{G\to H}$ when $G$ is a complete graph and $H$ is a star graph.

In Section \ref{sec2}, we study the ideal of weak graph homomorphisms of directed graphs (see Definition \ref{defdir}). As the first main result of that section, in Theorem \ref{thm:directedlinres}, we characterize when all (equivalently, some) powers of ideal of weak graph homomorphisms of directed graphs have linear resolutions. It turns out that for directed graphs $G$ and $H$ with $|V(H)|\geq 2$, the ideal $I_{G\to H}^k$ has a linear resolution if and only if either $E(G)=\emptyset$ or $G$ is an acyclic graph and $H$ is an acyclic complete graph. Note that a directed graph is said to be {\it acyclic} if it has no directed cycle as a subgraph. In the same theorem, we also prove that having a linear resolution and being linearly related are equivalent for $I_{G\to H}^k$. In Theorem \ref{regdirect}, we provide lower bounds for the projective dimension and the Castelnuovo-Mumford regularity of ordinary and symbolic powers of ideals of weak graph homomorphisms of directed graphs. The next goal of Section \ref{sec2} is to study unmixedness of ideals of weak graph homomorphisms. However, in general, it looks difficult to characterize all directed graphs $G$ and $H$ for which $I_{G\to H}$ is unmixed. So, we study this question only in a special case. Indeed, let $D_m$ denote the directed graph whose underlying graph is the complete graph on $[m]$ and $(i,j)\in E(G)$ if and only if $i< j$. We prove in Theorem \ref{unmixdir} that for any directed graph $G$, the ideal $I_{G\to D_m}$ is unmixed. Next, in Theorem \ref{pddirected}, we determine the projective dimension of $I_{G\to D_m}$. As a consequence, in Corollary \ref{CMdir}, we characterize all graphs $G$ for which $I_{G\to D_m}$ is Cohen-Macaulay. The last result of this paper, Theorem \ref{depthdecr} says that when $G$ and $H$ are directed graphs such that $I_{G\to H}$ is unmixed, then the depth function of symbolic powers of $I_{G\to H}$ is nonincreasing. 


\section{Preliminaries} \label{sec0}

In this section, we provide the definitions and the basic facts which will be used in the next section.

Let $G$ be a graph with vertex set $V(G)=[n]$ which can be directed or undirected. Then $G$ is called {\it totally disconnected} if $E(G)=\emptyset$. A subset $A\subseteq [n]$ is called an {\it independent set} of $G$ if there is no edge between the vertices in $A$. The size of the largest independent set of $G$ is the {\it independence number} of $G$, denoted by $\alpha(G)$. Two vertices $i,j\in V(G)$ are said to be neighbors if there is an edge of $G$ which is incident to both of these vertices. The set of all neighbors of $i\in V(G)$ is denoted by $N_G(i)$. Also, set $N_G[i]:=N_G(i)\cup\{i\}$. The graph obtained from $G$ by deleting a subset $U\subseteq V(G)$ is denoted by $G\setminus U$. When $U=\{i\}$ is a singleton, we write $G\setminus i$ instead of $G\setminus \{i\}$.

The complete undirected graph on $n$ vertices is denoted by $K_n$. The {\it star graph} $K_{1,n}$ is the undirected graph on $[n+1]$ with edge set $\big\{\{i,n+1\}: i=1, \ldots, n\big\}$.

Let $G$ be a directed graph on $[n]$ and suppose $i,j\in [n]$. A sequence$$i=v_1,e_1,v_2, e_2, \ldots, v_k, e_k, v_{k+1}=j$$of vertices and edges of $G$ is called a {\it directed walk} from $i$ to $j$ if for each $\ell=1, \ldots, k$, one has $e_{\ell}=(v_{\ell},v_{\ell+1})$. If the vertices of a directed walk are all distinct, then it is called a {\it directed path}. We say that a directed graph is {\it connected} if its underlying graph is connected.

Let $K$ be a field and $R= K[x_1,\ldots,x_t]$  be the
polynomial ring in $t$ variables over $K$. Suppose that $M$ is a graded $R$-module with minimal free resolution
$$0  \longrightarrow \cdots \longrightarrow  \bigoplus_{j}R(-j)^{\beta_{1,j}(M)} \longrightarrow \bigoplus_{j}R(-j)^{\beta_{0,j}(M)}   \longrightarrow  M \longrightarrow 0.$$
The {\it Castelnuovo--Mumford regularity} (or simply, regularity) and the {\it projective dimension} of $M$ are defined as
$$\reg M=\max\{j-i\mid \beta_{i,j}(M)\neq0\} \ \ \ \ \ \ {\rm and} \ \ \ \ \ \ \projdim M=\max\{i\mid \beta_{i,j}(M\neq 0 \ {\rm for \ some} \ j\}.$$

Let $I$ be a monomial ideal of $R$. The unique set of minimal monomial generators of $I$ will be denoted by $\cG(I)$. The LCM {\it lattice} of $I$ is the lattice of the least common multiples of all subsets of $\cG(I)$ (including empty set) ordered by divisibility. We know from \cite[Theorem 2.1]{gpw} that if two monomial ideals have the same LCM lattice, then they have the same projective dimension.

Assume that $I$ is a monomial ideal generated in a single degree $d$. Then we say that

$\bullet$ $I$ has a {\it linear resolution} if $\reg I=d$;

$\bullet$ $I$ is {\it linearly related} if $\beta_{1,j}(I)=0$, for each integer $j\neq d+1$;

$\bullet$ $I$ is a {\it polymatroidal ideal} if for any pair monomials $u=x_1^{a_1}\ldots x_t^{a_t}$ and $v = x_1^{b_1}\ldots x_t^{b_t}$ belonging to $\cG(I)$
and for every $i$ with $a_i > b_i$, there is an integer $j$ with $a_j < b_j$ such that
$x_j(u/x_i)\in \cG(I)$;

$\bullet$ $I$ is a {\it matroidal ideal} if it is a squarefree polymatroidal ideal;

$\bullet$ $I$ is a {\it transversal polymatroidal ideal} if it is a product of monomial prime ideals.

Let $I$ be a monomial ideal (which is not necessarily generated in a single degree). Then $I$ has {\it linear quotients} if there is a linear order $u_1\prec u_2\prec \cdots \prec u_k$ on $\cG(I)$ such that for each $i=2, \ldots, k$, the colon ideal $(u_1, \ldots, u_{i-1}):u_i$ is generated by a subset of variables. It is a standard fact that if a monomial ideal $I$ is generated in a single degree and has linear quotients, then it admits a linear resolution (see e.g., \cite[Proposition 8.2.1]{HHbook}).

A monomial ideal $I$ of $R$ is called {\it weakly polymatroidal} if for every two monomials $u=x_1^{a_1}\ldots x_t^{a_t}$ and $v=x_1^{b_1}\ldots x_t^{b_t}$ in $\cG(I)$ with $a_1=b_1, \ldots, a_{r-1}=b_{r-1}$ and $a_r > b_r$ for some $r$, there exists $s > r$ such that $x_r(v/x_s)\in I$. It is well-known that any weakly polymatroidal ideal has linear quotients (see e.g., \cite[Theorem 12.7.2]{HHbook}).

Let $I$ be a squarefree monomial ideal of $R$ and suppose that $I$ has the irredundant primary decomposition $I =\mathfrak{p}_1\cap\cdots\cap\mathfrak{p}_r$, where every $\mathfrak{p}_i$ is an ideal of $R$ generated by a subset of the variables. For each integer $k\geq 1$, the $k$th {\it symbolic power} of $I$ is the ideal$$I^{(k)}=\mathfrak{p}_1^k\cap\cdots\cap\mathfrak{p}_r^k.$$

A monomial ideal $I$ of $R$ is called {\it unmixed} if its associated primes have the same height. We say that $I$ is {\it Cohen-Macaulay} if the ring $R/I$ is Cohen-Macaulay, i.e., $\depth R/I=\dim R/I$. Moreover, $I$ is a {\it complete intersection} if $|\cG(I)|=\height I$. It is well-known that any complete intersection ideal is Cohen-Macaulay, and any Cohen-Macaulay ideal is unmixed.

\section{Undirected graphs} \label{sec1}
In this section, we study the ideal of weak graph homomorphisms of undirected graphs. Let $G$ and $H$ be simple graphs on vertex sets $[n]$ and $[m]$, respectively. Recall from Introduction that  the ideal of weak graph homomorphisms from $G$ to $H$ is
\[
I_{G\to H}=(u_\phi:\phi\in\WHom(G,H))\subseteq S,
\] 
where $S=K[x_{ij}: i\in [n], j\in [m]]$ and $u_\phi=\prod_{i=1}^{n}x_{i\phi(i)}$, for each $\phi\in \WHom(G,H)$.





The following observation describes $I_{G\to H}$ in two extremal cases.

\begin{Remark}\label{rem:1}
Let $G$ and $H$ be graphs on vertex sets $[n]$ and $[m]$, respectively. If
$G$ is totally disconnected or $H$ is a complete graph, then every
map $\phi:G\to H$ is a weak homomorphism. Consequently,
\[
I_{G\to H}=\prod_{i=1}^{n}\mathfrak{p}_i,
\qquad
\mathfrak{p}_i=(x_{ij}:j\in [m]).
\]
Since the ideals $\mathfrak{p}_1, \ldots, \mathfrak{p}_n$ are generated by variables in pairwise
disjoint sets of variables, the ideal $I_{G\to H}$ is a transversal
matroidal ideal.
\end{Remark}

The following criterion for linearly related monomial ideals will be used to characterize when $I_{G\to H}$ and its powers have linear resolutions.

Recall that the least common multiple of two monomials $u,v$ is denoted by $\lcm(u,v)$. Let $I$ be a monomial ideal generated in degree $d$. In \cite{BHZ2017}, the authors associated a graph $G_I$ to $I$ as follows: the vertex set of $G_I$ is $\cG(I)$, and $\{u,v\}\in E(G_I)$ if and only if $\deg(\lcm(u,v))=d+1$. Moreover, for $u,v\in \cG(I)$, let $G_I^{(u,v)}$ denote the induced subgraph of $G_I$ on the vertex set
\[
\{w\in \cG(I): w \mid \lcm(u,v)\}.
\]
The connectivity of $G_I^{(u,v)}$ characterizes when $I$ is linearly related.

\begin{Theorem}\cite[Corollary 2.2]{BHZ2017}\label{ref:BHN}
Let $I$ be a monomial ideal generated in degree $d$. Then $I$ is linearly related if and only if for all $u,v\in \cG(I)$ there is a path in $G_I^{(u,v)}$ connecting $u$ and $v$.
\end{Theorem}


We now characterize when the ideal $I_{G\to H}^k$ has a linear resolution.

\begin{Theorem}\label{Theorem:linres}
Let $G$ and $H$ be graphs. Then the following conditions are equivalent.
\begin{enumerate}
    \item $I_{G\to H}^k$ is a polymatroidal ideal, for all $k\geq 1$;
    \item $I_{G\to H}^k$ is a polymatroidal ideal, for some $k\geq 1$;
 
    \item $I_{G\to H}^k$ has linear quotients, for some $k\geq 1$;
    \item $I_{G\to H}^k$ has a linear resolution, for some $k\geq 1$;
    \item $I_{G\to H}^k$ is linearly related,  for some $k\geq 1$;
    \item $G$ is totally disconnected or $H$ is a complete graph.
\end{enumerate}
\end{Theorem}

\begin{proof}
Let $V(G)=[n]$ and $V(H)=[m]$. The implication $(1)\Rightarrow(2)$ is
trivial, and $(2)\Rightarrow(3)$ is well known, for example, see \cite[Theorem 12.6.2]{HHbook}. Since $I_{G\to H}^k$ is generated in degree $nk$, the implications $(3)\Rightarrow(4)\Rightarrow(5)$ are standard. Thus it suffices to prove
$(5)\Rightarrow(6)$ and $(6)\Rightarrow(1)$.

We first prove $(6)\Rightarrow(1)$. Assume that $G$ is totally disconnected
or that $H$ is a complete graph. By Remark~\ref{rem:1},
$I_{G\to H}=\prod_{i=1}^{n}\mathfrak{p}_i$, where $\mathfrak{p}_i=(x_{ij}:j\in[m])$. Hence, for
every $k\geq 1$,
\[
I_{G\to H}^k=\prod_{i=1}^{n}\mathfrak{p}_i^k.
\]
Since $I_{G\to H}^k$ is a product of monomial prime ideals, it is a transversal polymatroidal ideal.

Next we prove $(5)\Rightarrow(6)$ by contraposition. Suppose that $G$ is not
totally disconnected and that $H$ is not a complete graph. Then
there exist an edge $\{p,q\}\in E(G)$ and two distinct vertices $a,b\in V(H)$
such that $\{a,b\}\notin E(H)$.

Let $\phi$ and $\psi$ be the constant maps defined by $\phi(i)=a$ and
$\psi(i)=b$ for all $i\in [n]$. Then $\phi, \psi \in \WHom(G,H)$, and $u_\phi,u_\psi\in \cG(I_{G\to H})$. Thus, for any $k\geq 1$ the monomials $u_\phi^k$ and
$u_\phi^{k-1}u_\psi$ belong to $\cG(I_{G\to H}^k)$. We claim that there is no path connecting $u_\phi^k$ and
$u_\phi^{k-1}u_\psi$ in $G^{(u_\phi^k,u_\phi^{k-1}u_\psi)}_{I_{G\to H}^k}$.

Let $L$ be the connected component of $G$ containing the edge $\{p,q\}$, and
let $w$ be a vertex of
$G^{(u_\phi^k,u_\phi^{k-1}u_\psi)}_{I_{G\to H}^k}$. Then $w$ is a minimal
generator of $I_{G\to H}^k$ dividing
$\lcm(u_\phi^k,u_\phi^{k-1}u_\psi)$. Thus $w$ is a product
$u_{\theta_1}\cdots u_{\theta_k}$, where each $\theta_t\in\WHom(G,H)$ has
image contained in $\{a,b\}$. Since $\{a,b\}\notin E(H)$, each $\theta_t$ is
constant on $L$.

Moreover, because $w$ divides
$\lcm(u_\phi^k,u_\phi^{k-1}u_\psi)$, at most one of the maps
$\theta_t$ sends $L$ to $b$. Consequently, $w$ is divisible by exactly one of the monomials
\[
m_a=\prod_{i\in L}x_{ia}^{k} \quad \text{or} \quad
m_b=\prod_{i\in L}x_{ia}^{k-1}x_{ib}.
\]
Now let $w$ and $w'$ be adjacent vertices of
$G^{(u_\phi^k,u_\phi^{k-1}u_\psi)}_{I_{G\to H}^k}$. Then $\deg\lcm(w,w')=nk+1.$ Obviously, a generator divisible by $m_a$ and a generator divisible
by $m_b$ differ in at least $|V(L)|\geq 2$ variables. Therefore $w$ and $w'$ must be
 divisible by the same one of the monomials $m_a$ and $m_b$. However,
$u_\phi^k$ is divisible by $m_a$, whereas
$u_\phi^{k-1}u_\psi$ is divisible by $m_b$. Hence there is no path
connecting $u_\phi^k$ and $u_\phi^{k-1}u_\psi$ in $G^{(u_\phi^k,u_\phi^{k-1}u_\psi)}_{I_{G\to H}^k}$.
\end{proof}


  


The following decomposition generalizes Remark~\ref{rem:1} and will be used
repeatedly in the sequel.

\begin{Remark}\label{rem:disjoint-complete}
Let $G$ be a graph on $[n]$ with connected components $G_1,\ldots,G_s$, and suppose
that
\[
H=K_{m_1}\sqcup\cdots\sqcup K_{m_r}
\]
is a graph on $[m]$ which is a disjoint union of complete graphs. For $i\in V(G)$ and $k\in [r]$, set
\[
\mathfrak{p}_{i,k}=(x_{ij}:j\in V(K_{m_k})).
\]
Since every weak graph homomorphism from $G$ to $H$ maps each connected component
of $G$ into a single connected component of $H$, we have
\[
I_{G\to H}
=
\prod_{\ell=1}^{s}
\left(
\sum_{k=1}^{r}
\prod_{i\in V(G_\ell)}\mathfrak{p}_{i,k}
\right).
\]
Moreover, for distinct values of $k$, the ideals
$\prod_{i\in V(G_\ell)}\mathfrak{p}_{i,k}$ involve pairwise disjoint sets of variables. Moreover, for distinct values of $\ell$, the ideals
\[
\sum_{k=1}^{r}\prod_{i\in V(G_\ell)}\mathfrak{p}_{i,k}
\]
also involve pairwise disjoint sets of variables.
\end{Remark}

The following example illustrates the decomposition in Remark~\ref{rem:disjoint-complete}.



\begin{Example}
Let $G$ be the graph with edges $\{a,b\}$ and $\{c,d\}$. Then $G$ has two
connected components $G_1$ and $G_2$ with $V(G_1)=\{a,b\}$ and $V(G_2)=\{c,d\}.$ Let $H=K_1\sqcup K_2$, where $V(K_1)=\{p\}$ and $V(K_2)=\{q,r\}.$ Then
\[
\mathfrak{p}_{a,1}=(x_{ap}), \quad
\mathfrak{p}_{b,1}=(x_{bp}), \quad
\mathfrak{p}_{c,1}=(x_{cp}), \quad
\mathfrak{p}_{d,1}=(x_{dp}),
\]
and
\[
\mathfrak{p}_{a,2}=(x_{aq},x_{ar}), \quad
\mathfrak{p}_{b,2}=(x_{bq},x_{br}), \quad
\mathfrak{p}_{c,2}=(x_{cq},x_{cr}), \quad
\mathfrak{p}_{d,2}=(x_{dq},x_{dr}).
\]
Hence, by Remark~\ref{rem:disjoint-complete},
\[
I_{G\to H}
=
\bigl(\mathfrak{p}_{a,1}\mathfrak{p}_{b,1}+\mathfrak{p}_{a,2}\mathfrak{p}_{b,2}\bigr)
\bigl(\mathfrak{p}_{c,1}\mathfrak{p}_{d,1}+\mathfrak{p}_{c,2}\mathfrak{p}_{d,2}\bigr).
\]
Equivalently,
\[
I_{G\to H}
=
\bigl((x_{ap}x_{bp})+(x_{aq},x_{ar})(x_{bq},x_{br})\bigr)
\bigl((x_{cp}x_{dp})+(x_{cq},x_{cr})(x_{dq},x_{dr})\bigr).
\]
\end{Example}
The following proposition computes the height of $I_{G\to H}$ and will be used
in the proofs of the unmixedness and Cohen--Macaulay criteria.

\begin{Proposition}\label{prop:height}
Let $G$ and $H$ be graphs. Then
\[
\height I_{G\to H}=|V(H)|.
\]
\end{Proposition}

\begin{proof}
For each $i\in V(G)$, let $\mathfrak{p}_i=(x_{ij}:j\in V(H)).$ Since every generator $u_\phi$ of $I_{G\to H}$ is divisible by exactly one
variable of the form $x_{i\phi(i)}$, we have $I_{G\to H}\subseteq \mathfrak{p}_i$. Therefore $\height I_{G\to H}\leq |V(H)|.$

Conversely, let $\mathfrak{p}$ be a prime monomial ideal containing $I_{G\to H}$. For
each $j\in V(H)$, the constant map $\phi_j(v)=j$ is a weak graph homomorphism, and
\[
u_{\phi_j}=\prod_{i\in V(G)}x_{ij}\in I_{G\to H}\subseteq \mathfrak{p}.
\]
Since $\mathfrak{p}$ is monomial prime of $I_{G\to H}$, for each $j$, it must contain at least one variable from the
set $\{x_{ij}:i\in V(G)\}$. Hence $\height \mathfrak{p}\geq |V(H)|$, and therefore $\height I_{G\to H}=|V(H)|.$
\end{proof}


We next characterize when the ideal $I_{G\to H}$ is unmixed.

\begin{Theorem}\label{theorem:unmixed}
The ideal $I_{G\to H}$ is unmixed if and only if $G$ is totally disconnected
or $H$ is a disjoint union of complete graphs.
\end{Theorem}

\begin{proof}
$\Leftarrow)$ Let $V(G)=[n]$ and $V(H)=[m]$. Assume first that $G$ is totally disconnected. By
Remark~\ref{rem:1},
\[
I_{G\to H}=\prod_{i=1}^{n}\mathfrak{p}_i,
\qquad
\mathfrak{p}_i=(x_{ij}:j\in [m]).
\]
Since the ideals $\mathfrak{p}_1, \ldots, \mathfrak{p}_n$ are monomial prime ideals generated by
variables in pairwise disjoint sets of variables, we have
\[
I_{G\to H}
=
\prod_{i=1}^{n}\mathfrak{p}_i
=
\bigcap_{i=1}^{n}\mathfrak{p}_i.
\]
Hence $\Min(I_{G\to H})=\{\mathfrak{p}_1,\ldots,\mathfrak{p}_n\}$. Since each $\mathfrak{p}_i$ has height $m$,
it follows that $I_{G\to H}$ is unmixed.

Now suppose that $H=K_{m_1}\sqcup\cdots\sqcup K_{m_r}$ is a disjoint union of complete graphs. Using the notation of
Remark~\ref{rem:disjoint-complete}, write
\[
I_{G\to H}=\prod_{\ell=1}^{s}J_\ell,
\qquad
J_\ell=
\sum_{k=1}^{r}\prod_{i\in V(G_\ell)}\mathfrak{p}_{i,k}.
\]

We first determine the minimal primes of $J_\ell$. Since, for each $i$ and
$k$, the ideals $\mathfrak{p}_{i,k}$ are monomial prime ideals generated by variables in
pairwise disjoint sets of variables, we have
\[
\prod_{i\in V(G_\ell)}\mathfrak{p}_{i,k}
=
\bigcap_{i\in V(G_\ell)}\mathfrak{p}_{i,k}.
\]
Thus, for fixed $k$, the ideal
$\prod_{i\in V(G_\ell)}\mathfrak{p}_{i,k}$ has minimal primes $\mathfrak{p}_{i,k}$ with
$i\in V(G_\ell)$. Since the ideals
$\prod_{i\in V(G_\ell)}\mathfrak{p}_{i,k}$ involve pairwise disjoint sets of variables
for distinct values of $k$, the minimal primes of $J_\ell$ are precisely the ideals
\[
\mathfrak{p}_{i_1,1}+\cdots+\mathfrak{p}_{i_r,r},
\qquad
i_k\in V(G_\ell).
\]
Each such prime has height $m_1+\cdots+m_r=|V(H)|$. Hence every $J_\ell$ is
unmixed of height $|V(H)|$.

Since the ideals $J_1,\ldots,J_s$ involve pairwise disjoint sets of variables,
the minimal primes of $I_{G\to H}$ are precisely the minimal primes of the
ideals $J_\ell$, where $\ell=1,\ldots,s$. Therefore every minimal prime of
$I_{G\to H}$ has height $|V(H)|$, and so $I_{G\to H}$ is unmixed.\\

$\Rightarrow)$ Conversely, we argue by contraposition. Suppose that $G$ is not totally
disconnected and that $H$ is not a disjoint union of complete
graphs. We construct a minimal prime $\mathfrak{p}$ of $I_{G\to H}$ such that
$\height \mathfrak{p}>|V(H)|$. Since $\height I_{G\to H}=|V(H)|$ by
Proposition~\ref{prop:height}, it follows that $I_{G\to H}$ is not unmixed.

Since $H$ is not a disjoint union of complete graphs, there exist vertices
$a,b\in V(H)$ belonging to the same connected component such that
$\{a,b\}\notin E(H)$. We may choose $a$ and $b$ so that they have a common
neighbor $c\in V(H)$. Thus $\{a,c\},\{b,c\}\in E(H)$ and $\{a,b\}\notin E(H).$ Since $G$ is not totally disconnected, there exist
vertices $p,q\in V(G)$ such that $\{p,q\}\in E(G)$.

Let
\[
A=\{x\in N_H[a]:N_H[x]\subseteq N_H[a]\}.
\]
Set
\[
\mathfrak{p}:=(x_{pi},x_{qj}:i\in V(H)\setminus A,\ j\in N_H[a]).
\]

We claim that $\mathfrak{p}\in\Min(I_{G\to H})$. First we show that
$I_{G\to H}\subseteq \mathfrak{p}$. Let $\phi\in\WHom(G,H)$. If
$\phi(q)\in N_H[a]$, then $x_{q\phi(q)}\in \mathfrak{p}$, and hence $u_\phi\in \mathfrak{p}$. If
$\phi(q)\notin N_H[a]$, then, since $\{p,q\}\in E(G)$, we have
$\phi(p)\in N_H[\phi(q)]$. Consequently $\phi(p)\notin A$, as otherwise
$N_H[\phi(p)]\subseteq N_H[a]$, contradicting
$\phi(q)\notin N_H[a]$. Thus $x_{p\phi(p)}\in \mathfrak{p}$, and again $u_\phi\in \mathfrak{p}$.

Next we show that the ideal obtained from $\mathfrak{p}$ by deleting any generator does not contain $I_{G\to H}$. Let $j\in N_H[a]$
and define $\phi:G\to H$ by $\phi(p)=a$ and $\phi(v)=j$ for all $v\neq p$.
Then the only variable from $\mathfrak{p}$ dividing $u_\phi$ is $x_{qj}$. Hence the ideal obtained from $\mathfrak{p}$ by deleting $x_{qj}$ does not contain $I_{G\to H}$.

Now let $i\in V(H)\setminus A$. If $i\notin N_H[a]$, then the constant map
$\phi(v)=i$ is a weak graph homomorphism, and the only variable from $\mathfrak{p}$ dividing $u_\phi$ is $x_{pi}$. Suppose instead that
$i\in N_H[a]\setminus A$. Then $N_H[i]\nsubseteq N_H[a]$, so there exists
$d\in N_H[i]\setminus N_H[a]$. Define $\phi:G\to H$ by $\phi(p)=i$ and
$\phi(v)=d$ for all $v\neq p$. Since $d\notin N_H[a]$, the only variable from
$\mathfrak{p}$ dividing $u_\phi$ is $x_{pi}$. Hence the ideal obtained from $\mathfrak{p}$ by deleting any generator does not contain $I_{G\to H}$. Therefore $\mathfrak{p}\in\Min(I_{G\to H})$.

Finally,
\[
\height \mathfrak{p}
=
|V(H)\setminus A|+|N_H[a]|
=
|V(H)|+\bigl(|N_H[a]|-|A|\bigr).
\]
Since $c\in N_H[a]$ and $b\in N_H[c]\setminus N_H[a]$, we have $c\notin A$.
Hence $A\subsetneq N_H[a]$, and therefore $\height \mathfrak{p}>|V(H)|$, proving that $I_{G\to H}$ is not unmixed.
\end{proof}



We now characterize the Cohen--Macaulay and complete intersection properties
of $I_{G\to H}$. The next theorem shows that, for ideals of weak graph homomorphisms, the
Cohen--Macaulay and complete intersection properties coincide.

\begin{Theorem}\label{thm:CM+CI}
Let $G$ and $H$ be two graphs. Then the following conditions are equivalent.
\begin{enumerate}
  \item $I_{G\to H}$ is a complete intersection.
    \item $I_{G\to H}$ is Cohen--Macaulay.
    \item One of the following conditions holds.
    \begin{enumerate}
        \item $G$ consists of an isolated vertex;
        \item $G$ is connected and $H$ is totally disconnected;
        \item $G$ is disconnected and $H$ consists of an isolated vertex.
    \end{enumerate}
\end{enumerate}
\end{Theorem}

\begin{proof}
It is well-known that every complete intersection is Cohen--Macaulay. Hence $(1)\Rightarrow(2)$. 

We first prove $(3)\Rightarrow(1)$. If $G$ consists of an isolated vertex, say 1, then
$I_{G\to H}=(x_{1j}:j\in V(H))$, which is generated by variables. Hence
$I_{G\to H}$ is a complete intersection.

If $G$ is connected and $H$ is totally disconnected, then every weak graph
homomorphism from $G$ to $H$ is constant. Therefore
\[
I_{G\to H}
=
\left(\prod_{i\in V(G)}x_{ij}:j\in V(H)\right).
\]
The generators have pairwise disjoint supports, and hence form a regular
sequence. Thus $I_{G\to H}$ is a complete intersection.

Finally, if $G$ is disconnected and $H$ consists of a single vertex, then
there is only one weak homomorphism from $G$ to $H$. Hence $I_{G\to H}$ is
principal, and therefore a complete intersection.

It remains to prove $(2)\Rightarrow(3)$. Let $V(G)=[n]$ and $V(H)=[m]$. Assume that $I_{G\to H}$ is
Cohen--Macaulay. Then $I_{G\to H}$ is unmixed. By
Theorem~\ref{theorem:unmixed}, either $G$ is totally disconnected or $H$ is a
disjoint union of complete graphs.

First suppose that $G$ is totally disconnected. By
Remark~\ref{rem:1}, we have $I_{G\to H}=\prod_{i=1}^{n}\mathfrak{p}_i$, where
$\mathfrak{p}_i=(x_{ij}:j\in V(H))$, and $I_{G\to H}$ is a transversal matroidal ideal. If $n=1$, then condition $(3a)$ holds. So, assume that $n\geq 2$. If $m=1$, then $(3c)$ holds. Therefore, suppose that $m\geq 2$. By \cite[Theorem 12.6.7]{HHbook}, a
Cohen--Macaulay matroidal ideal is either a principal or a squarefree Veronese
ideal. Since $I_{G\to H}=\prod_{i=1}^{n}\mathfrak{p}_i$ is a product of monomial prime
ideals generated in pairwise disjoint sets of variables, it cannot be a
squarefree Veronese ideal when $m,n\geq 2$. Hence $I_{G\to H}$ is principal.
This happens only when each $\mathfrak{p}_i$ is principal, equivalently $m=1$, a contradiction. 

Now suppose that $G$ is not totally disconnected. Then $H$ is a disjoint union
of complete graphs. Write $H=K_{m_1}\sqcup\cdots\sqcup K_{m_r}$,
and let $G_1,\ldots,G_s$ be the connected components of $G$. Set
$n_\ell=|V(G_\ell)|$. By
Remark~\ref{rem:disjoint-complete},
\[
I_{G\to H}=\prod_{\ell=1}^{s}J_\ell, \quad \text{ where } \quad
J_\ell=
\sum_{k=1}^{r}\prod_{i\in V(G_\ell)}\mathfrak{p}_{i,k},
\quad \text{ and } \quad
\mathfrak{p}_{i,k}=(x_{ij}:j\in V(K_{m_k})).
\]

We compute the depth of $S/I_{G\to H}$. For fixed $\ell$, let $S_\ell$ be the
polynomial ring generated by the variables appearing in $J_\ell$. By \cite[Lemma 7.3.7]{Villareal}, we have 
\[
\depth S_\ell/J_\ell
=
\sum_{k=1}^{r}
\depth \left(S_{\ell,k}\Big/\prod_{i\in V(G_\ell)}\mathfrak{p}_{i,k}\right),
\]
where $S_{\ell,k}$ is the polynomial ring generated by the variables appearing
in $\prod_{i\in V(G_\ell)}\mathfrak{p}_{i,k}$. For each $k$, the ideal
$\prod_{i\in V(G_\ell)}\mathfrak{p}_{i,k}$ is a transversal matroidal ideal of degree
$n_\ell$. Hence, by \cite[Theorem 2.5]{CHW}, its quotient has depth
$n_\ell-1$. Therefore $\depth S_\ell/J_\ell=r(n_\ell-1)$.

Since the ideals $J_1,\ldots,J_s$ involve pairwise disjoint sets of variables, it follows from \cite[Lemma 2.2]{HT} that
\[
\depth S/I_{G\to H}
=
\sum_{\ell=1}^{s}\depth S_\ell/J_\ell+(s-1)
=
r(n-s)+s-1.
\]

On the other hand, by Proposition~\ref{prop:height}, $\height I_{G\to H}=m$.
Since $S$ has $nm$ variables, $\dim S/I_{G\to H}=nm-m$. As
$I_{G\to H}$ is Cohen--Macaulay, we get
\begin{equation}\label{eq:cm}
r(n-s)+s-1=nm-m.
\end{equation}

We now distinguish two cases. If $G$ is connected, then $s=1$, and
(\ref{eq:cm}) becomes $r(n-1)=m(n-1)$. Since $G$ is not totally disconnected,
we have $n>1$. Thus $r=m$, which means that every connected component of $H$
has exactly one vertex. Hence $H$ is totally disconnected, and condition
$(3b)$ holds.

If $G$ is disconnected, then $s>1$. Since $G$ is not totally disconnected, we
have $n>s$. We claim that $m=1$. By (\ref{eq:cm}),
\[
0
=
nm-m-\bigl(r(n-s)+s-1\bigr)
=
(m-r)n+s(r-1)-m+1.
\]
Since $m\geq r$ and $n>s$, we obtain
\[
0\geq (m-r)s+s(r-1)-m+1=(m-1)(s-1).
\]
As $s>1$, it follows that $m=1$. Hence $H$ consists of a single vertex, and condition $(3c)$ holds.
\end{proof}

The next proposition shows that when $I_{G\to H}$ is unmixed, then its $k$th ordinary and symbolic powers coincide.

\begin{Proposition} \label{symbord}
Let $G$ and $H$ be graphs such that $I_{G\to H}$ is unmixed. Then for every integer $k\geq 1$, one has $I_{G\to H}^{(k)}=I_{G\to H}^k$.    
\end{Proposition}

\begin{proof}
We know from Theorem~\ref{theorem:unmixed} that either $G$ is totally disconnected or $H$ is a
disjoint union of complete graphs. In the first case, by Remark~\ref{rem:1}, the ideal $I_{G\to H}$ is a product of monomial prime ideals involving disjoint sets of variable. In the second case, by Remark~\ref{rem:disjoint-complete}, the ideal $I_{G\to H}$ is obtained by the sum and product of squarefree monomial ideals involving disjoint sets of variables. In both cases, the assertion follows from the following observations.

$\bullet$ Let $I$ ad $J$ be squarefree monomial ideals on disjoint sets of variables such that $I^{(k)}=I^k$, for each $k\geq 1$. Then$$(IJ)^k=I^kJ^k=I^k\cap J^k=I^{(k)}\cap J^{(k)}=(I\cap J)^{(k)}=(IJ)^{(k)}.$$

$\bullet$ Let $I$ ad $J$ be squarefree monomial ideals on disjoint sets of variables such that $I^{(k)}=I^k$, for each $k\geq 1$. Then
\begin{align*}
(I+J)^k=\sum_{i=0}^kI^iJ^{k-i}=\sum_{i=0}^kI^{(i)}J^{(k-i)}=(I+J)^{(k)},   
\end{align*}
where the last equality follows from \cite[Theorem 3.4]{HNTT}.
\end{proof}

We next study the behavior of regularity and projective dimension of the ideals of weak graph homomorphisms. We first show that the Betti numbers of these ideals are monotone under passing to induced subgraphs of $H$. We mention that for any monomial ideal $I\subseteq S$ and for any monomial $M\in S$, the monomial ideal generated by $\{u\in I\mid u \ {\rm divides} \ M\}$ is denoted by $I^{\leq M}$.




\begin{Lemma}\label{restriction}
Let $G$ and $H$ be graphs and let $H'$ be an induced subgraph of $H$. Then for any $i,j\geq 0$ and for any integer $k\geq 1$, one has
\begin{itemize}
    \item [(i)] $\beta_{i,j}(I_{G\rightarrow H'}^k)\leq\beta_{i,j}(I_{G\rightarrow H}^k)$.
    \item [(ii)] $\beta_{i,j}(I_{G\rightarrow H'}^{(k)})\leq\beta_{i,j}(I_{G\rightarrow H}^{(k)})$.
\end{itemize}
\end{Lemma}

\begin{proof}
Set
$M:=\prod_{i\in V(G)}\prod_{j\in V(H')}x_{ij}$. It is easy to see that
\[
I_{G\to H'}^k=(I_{G\to H}^k)^{\leq M^k} \ \ \ \ \ {\rm and} \ \ \ \ \ I_{G\to H'}^{(k)}=(I_{G\to H}^{(k)})^{\leq M^k}.
\]

By \cite[Lemma~4.4]{HHZ}, the minimal multigraded free resolution of
$(I_{G\to H}^k)^{\leq M^k}$ (resp. $(I_{G\to H}^{(k)})^{\leq M^k}$) is a subcomplex of the minimal multigraded free
resolution of $I_{G\to H}^k$ (resp. $(I_{G\to H}^{(k)})$). Therefore the assertions follow.
\end{proof}

Using above lemma, we obtain lower bounds for the regularity of $I^k_{G\to H}$ and $I^{(k)}_{G\to H}$ when $G$ is a connected graph. Recall that the independence number of a graph $G$ is denoted by $\alpha(G)$.


\begin{Lemma} \label{lem:regconundir}
Let $G$ and $H$ be graphs such that $G$ is connected. Then for any integer $k\geq 1$, one has
\begin{itemize}
    \item [(i)] $\reg I_{G\rightarrow H}^k\geq k|V(G)|+(|V(G)|-1)(\alpha(H)-1)$.
    \item [(ii)] $\reg I_{G\rightarrow H}^{(k)}\geq k|V(G)|+(|V(G)|-1)(\alpha(H)-1)$.
\end{itemize}
\end{Lemma}

\begin{proof}
Assume that $V(G)=[n]$ and $V(H)=[m]$ and set $\alpha:\alpha(H)$. Let $A=\{j_1,\ldots,j_{\alpha}\}$ be an independent set of $H$, and let $H'$ be the
induced subgraph of $H$ on $A$. Since $A$ is independent and $G$ is connected, every weak homomorphism from $G$ to $H'$ is constant. Hence
\[
I_{G\to H'}=(u_1,\ldots,u_{\alpha}),
\qquad
u_t=\prod_{i=1}^{n}x_{ij_t}.
\]
The monomials $u_1,\ldots,u_{\alpha}$ have pairwise disjoint supports and degree
$n$. Thus $I_{G\to H'}$ is a complete intersection generated by $\alpha$ monomials
of degree $n$. In particular, $I_{G\rightarrow H'}^k=I_{G\rightarrow H'}^{(k)}$. Therefore, \cite[Lemma 4.4]{BHT} implies that$$\reg I_{G\rightarrow H'}^k=\reg I_{G\rightarrow H'}^{(k)}\geq kn+(n-1)(\alpha-1).$$The assertion now follows from Lemma \ref{restriction}.
\end{proof}

In the following theorem, we extend Lemma \ref{lem:regconundir} to the case that $G$ is not necessarily connected.

\begin{Theorem} \label{thm:regundir}
Let $G$ and $H$ be graphs and let $c$ denote the number of connected components of $G$. Then for any integer $k\geq 1$, one has
\begin{itemize}
    \item [(i)] $\projdim I_{G\rightarrow H}^k\geq |V(H)|-1$.
    \item [(ii)] $\projdim I_{G\rightarrow H}^{(k)}\geq |V(H)|-1$.
    \item [(iii)] $\reg I_{G\rightarrow H}^k\geq k|V(G)|+(|V(G)|-c)(\alpha(H)-1)$.
    \item [(iv)] $\reg I_{G\rightarrow H}^{(k)}\geq k|V(G)|+(|V(G)|-c)(\alpha(H)-1)$.
\end{itemize}
\end{Theorem}

\begin{proof}
The assertions about projective dimension follow from Proposition \ref{prop:height}. So, we only prove parts (iii) and (iv). 

Let $G_1, \ldots, G_c$ denote the connected components of $G$. We use induction on $c$. For $c=1$, the assertions follow from Lemma \ref{lem:regconundir}. So, suppose $c\geq 2$. Let $G'$ denote the disjoint union of $G_2, \ldots, G_c$. In other words, $G$ is the disjoint union of $G_1$ and $G'$. It is easy to see that$$I_{G\rightarrow H}=I_{G_1\rightarrow H}I_{G'\rightarrow H}=I_{G_1\rightarrow H}\cap I_{G'\rightarrow H}.$$As a consequence,$$I_{G\rightarrow H}^k=I_{G_1\rightarrow H}^kI_{G'\rightarrow H}^k \ \ \ \ \ \ \ \ {\rm and} \ \ \ \ \ \ \ \ I_{G\rightarrow H}^{(k)}=I_{G_1\rightarrow H}^{(k)}\cap I_{G'\rightarrow H}^{(k)}=I_{G_1\rightarrow H}^{(k)} I_{G'\rightarrow H}^{(k)}.$$Since the minimal monomial generators of $I_{G'\rightarrow H}^k$ and $I_{G_1\rightarrow H}^k$ have disjoint sets of variables, the assertions follow from \cite[Lemma 3.2]{HT} and the induction hypothesis.
\end{proof}





Note that for $k=1$, the bound given for the projective dimension in Theorem~\ref{thm:regundir} is sharp, as the equality $\projdim I_{G\rightarrow H}^k=|V(H)|-1$ is equivalent to Cohen-Macaulayness of $I_{G\to H}$ which is already characterized in Theorem \ref{thm:CM+CI}. Theorem \ref{regstar} below shows that the bound for the regularity is also sharp, even for connected graphs. More precisely, for $k=1$, the equality holds in Theorem~\ref{thm:regundir} (iii) when $G$ is a complete graph and $H$ is a star graph.  To prove this, we first need the following lemma. Recall that a vertex $x$ of a graph $H$ is called a {\it simplicial vertex} if all its neighbors are adjacent.

\begin{Lemma} \label{delsimpl}
Let $G=K_n$ be the complete graph of $n$ vertices. Also, let $H$ be a graph on $[m]$ and assume that $1\in V(H)$ is a simplicial vertex of $H$. Then
$$\reg I_{G\to H}\leq \reg I_{G\to H\setminus 1}+(n-1).$$
\end{Lemma}

\begin{proof}
Set $I:=I_{G\to H}$ and $I_1:=I_{G\to H\setminus 1}$. Moreover, let $H'$ denote the induced subgraph of $H$ on $N_H[1]$. Since $1\in V(H)$ is a simplicial vertex, it follows that $H'$ is a complete graph. Set $I_2:=I_{G\to H'}$. We show that $I=I_1+I_2$. Indeed, it is clear that $I_1+I_2$ is contained in $I$. To prove the reverse inclusion, let $\phi\in \WHom(G,H)$ be such that $u_{\phi}\notin I_1$. Thus, there is a vertex $j\in [n]$ with $\phi(j)=1$. Since $G$ is a complete graph, it follows that for each vertex $i\in V(G)$, one has $\phi(i)\in N_H[1]$. In other words, $u_{\phi}\in I_2$. As a consequence, $I=I_1+I_2$. Since $H'$ is a complete graph, it follows from Theorem \ref{Theorem:linres} that $\reg I_2=n$. The assertion now follows from \cite[Corollary 3.2]{H}.
\end{proof}

Combining Theorem~\ref{thm:regundir} with the previous lemma, we obtain the following corollary.

\begin{Corollary} \label{regstar}
Let $s,n\geq 1$ be integers and $H=K_{1,s}$ be a star graph. Then $\reg I_{K_n\to H} =sn-s+1$.
\end{Corollary}

\begin{proof}
We know from Theorem \ref{thm:regundir} that $\reg I_{K_n\to H}\geq sn-s+1$. We prove that reverse inequality by induction on $s$. For $s=1$, we have $H=K_2$ and the assertion follows from Theorem \ref{Theorem:linres}. So, suppose that $s\geq 2$. Then Lemma \ref{delsimpl} implies that$$\reg I_{K_n\to H}\leq \reg I_{K_n\to K_{1,s-1}}+(n-1).$$Thus, the assertion follows from the above inequality and the induction hypothesis. 
\end{proof}

We close this section by proposing the following problem.

\begin{Problem}
Characterize all connected graphs $G$ and $H$ for which$$\reg I_{G\to H} =\alpha(H)(|V(G)|-1)+1.$$
\end{Problem}

\section{Directed graphs} \label{sec2}
In this section, we study the ideal of weak graph homomorphisms of directed graphs. 

\begin{Definition} \label{defdir}
Let $G$ and $H$ be directed graphs on vertex sets $[n]$ and $[m]$, respectively. A {\it weak graph homomorphism} from $G$ to $H$ is a map
$\phi:V(G)\to V(H)$ such that for every $(x,y)\in E(G)$, either
$\phi(x)=\phi(y)$ or $(\phi(x),\phi(y))\in E(H)$.

As before, let $\WHom(G,H)$ denote the set of all weak graph homomorphisms from $G$ to $H$, and
let $S=K[x_{ij}: i\in[n],\, j\in[m]]$. For each $\phi\in\WHom(G,H)$, define
\[
u_\phi=\prod_{i=1}^{n}x_{i\phi(i)}.
\]
The \emph{ideal of weak graph homomorphisms from $G$ to $H$} is
\[
I_{G\to H}=(u_\phi:\phi\in\WHom(G,H)).
\]
\end{Definition}

Unless otherwise stated, in this section, the graphs $G$ and $H$ are simple, i.e., their underlying graphs have no loops and multiple edges. 

As the first main result of this section, in Theorem \ref{thm:directedlinres}, we characterize all directed graphs $G$ and $H$ for which $I_{G\to H}$ and it powers have linear resolutions. We first need the following two lemmas. Recall that a directed graph is acyclic if it has no directed cycle as a subgraph.

\begin{Lemma} \label{dircyc}
Let $C$ denote a directed cycle on $[n]$ and lat $H$ be an acyclic directed graph. Then for each $\phi\in \WHom(G,H)$, one has $\phi(1)=\phi(2)=\cdots =\phi(n)$.    
\end{Lemma}

\begin{proof}
Without loss of generality, we may assume that$$E(C)=\{(1,2), (2,3), \ldots, (n-1,n), (n,1)\}.$$By symmetry, we only show that $\phi(1)=\phi(2)$. Since in $(1,2)\in E(C)$, it follows that in $H$, there must be a directed path $P$ of length at most one from $\phi(1)$ to $\phi(2)$. Similarly, since in $C$ there is a directed path from $2$ to $1$, in $H$ there must be a directed walk $Q$ from $\phi(2)$ to $\phi(1)$. As $H$ is an acyclic graph, it follows that the lengths of $P$ and $Q$ are zero. In particular, $\phi(1)=\phi(2)$.   
\end{proof}

\begin{Lemma} \label{compacy}
Let $H$ be a directed acyclic complete graph on $[m]$. Then there is a relabeling of vertices of $H$ such that $(i,j)$ is an edge of $H$ if and only if $i<j$.    
\end{Lemma}

\begin{proof}
We use induction on $m$. There is nothing to prove for $m=1$. So, suppose that $m\geq 2$. Since $H$ is an acyclic complete graph, there is a vertex $x\in V(H)$ such that for any $y\in V(H)\setminus \{x\}$, one has $(y,x)\in H$. We label this vertex by $m$. Then by induction hypothesis, $H\setminus m$ admits the desired labeling, completing the proof. 
\end{proof}

Recall from Introduction that $D_m$ is the directed graph whose underlying graph is the complete graph $K_m$, and $(i,j)$ is an edge of $D_m$ if and only if $i<j$.

\begin{Definition}
Let $G$ be an acyclic directed graph on $[n]$. Then the transitive closure of $G$ is defined to be the new graph with vertex set $[n]$ and edge set$$\{(i,j)\mid i\neq j \ {\rm and\ there \ is \ a \ directed \ path \ from} \ i \ {\rm to} \ j \ {\rm in} \  G\}.$$Since $G$ is acyclic, its transitive closure is a simple directed graph. 
\end{Definition}

\begin{Example} \label{exampath}
Let $P_n$ be the directed path on $[n]$. In other words,$$E(P_n)=\{(i,i+1)\mid 1\leq i\leq n-1\}.$$Then its transitive closure is $D_n$. 
\end{Example}

Note that if $G$ is an acyclic directed graph on $[n]$ with transitive closure $G^{\ast}$, then there is a poset $P$ on $[n]$ such that $i<_P j$ if and only if $(i,j)\in E(G^{\ast})$. For instance, in Example \ref{exampath}, the associated poset is a chain on $[n]$.

Now we state and prove the first main result of this section. Note that if $|V(H)|=1$, then $I_{G\to H}$ is a principal ideal. So, without loss of generality, in the following theorem, we assume that $|V(H)|\geq 2$.

\begin{Theorem}\label{thm:directedlinres}
Let $G$ and $H$ be directed graphs on $[n]$ and $[m]$, respectively. Assume that $m\geq 2$. Then the following conditions are equivalent.

\begin{itemize}
    \item [(i)] $I_{G\to H}^k$ is weakly polymatroidal for all $k\geq 1$.
    \item [(ii)] $I_{G\to H}^k$ is weakly polymatroidal for some $k\geq 1$.
    \item [(iii)] $I_{G\rightarrow H}^k$ has linear quotients for some $k\geq 1$.
    \item [(iv)] $I_{G\to H}^k$ has a linear resolution for some $k\geq 1$.
    \item [(v)] $I_{G\to H}^k$ is linearly related for some $k\geq 1$.
    \item [(vi)] $G$ is totally disconnected, or $G$ is an acyclic graph and $H$ is an acyclic complete graph.
\end{itemize} 
\end{Theorem}

\begin{proof}
The implications (i)$\Rightarrow$(ii) is obvious. The implications (ii)$\Rightarrow$(iii)$\Rightarrow$(iv)$\Rightarrow$(v) are standard facts. Now, we show that (v) implies (vi). So, suppose that $I_{G\rightarrow H}^k$ is linearly related for some $k\geq 1$. We first show that $G$ is acyclic. By contradiction, suppose that $G$ is not acyclic. Therefore, $G$ has a directed cycle, say $C$, as a subgraph. Without loss of generality, we may assume that $V(C)=[\ell]$, for some integer $\ell$ with $3\leq\ell\leq n$. It follows from Lemma \ref{dircyc} that for each weak homomorphism $\phi\in \WHom(G,H)$, one has $\phi(1)=\phi(2)=\cdots =\phi(\ell)$. Let $\phi_1 : G \to H$ (resp. $\phi_2 : G \to H$) be the weak homomorphism defined by $\phi_1(j)=1$ (resp. $\phi_2(j)=2$), for each $j\in [n]$.  Note that $u_{\phi_1}^k$ and $u_{\phi_1}^{k-1}u_{\phi_2}$ belong to the set of minimal monomial generators of $I_{G\rightarrow H}^k$. Consider the relation graph $G_{I_{G\rightarrow H}^k}^{(u_{\phi_1}^k,u_{\phi_1}^{k-1}u_{\phi_2})}$. Assume that $v$ is a vertex of $G_{I_{Gto H}^k}^{(u_{\phi_1}^k,u_{\phi_1}^{k-1}u_{\phi_2})}$ which belongs to the same connected component as $u_{\phi_1}^k$. Since $m\geq 2$, we deduce from Lemma \ref{dircyc} that $v$ is divisible by $x_{11}^kx_{21}^k\cdots x_{\ell1}^k$. Consequently, $u_{\phi_1}^{k-1}u_{\phi_2}$ does not belong to the same connected component as $u_{\phi_1}^k$. In particular, $G_{I_{G\rightarrow H}^k}^{(u_{\phi_1}^k,u_{\phi_1}^{k-1}u_{\phi_2})}$ is not connected. Hence Theorem \ref{ref:BHN} implies that $I_{G\rightarrow H}^k$ is not linearly related, a contradiction.


If $E(G)=\emptyset$, then we are done. So, assume that $E(G)\neq \emptyset$. We show that $H$ is the complete graph. By contradiction, assume that $H$ is not the complete graph and let $p,q\in V(H)$ be vertices of $H$ with $(p,q), (q,p)\notin E(H)$. Let $\psi_1 : G \to H$ (resp. $\psi_2 : G \to H$) be the weak homomorphism defined by $\phi_1(j)=p$ (resp. $\phi_2(j)=q$), for each $j\in [n]$. Note that $u_{\psi_1}^k$ and $u_{\psi_1}^{k-1}u_{\psi_2}$ belong to the set of minimal monomial generators of $I_{G\rightarrow H}^k$. Consider the relation graph $G_{I_{G\rightarrow H}^k}^{(u_{\psi_1}^k,u_{\psi_1}^{k-1}u_{\psi_2})}$. It is obvious that every vertex of this graph is divisible by $x_{1p}^{k-1}\cdots x_{np}^{k-1}$ and divides $x_{1p}^k\cdots x_{np}^kx_{1q}\cdots x_{nq}$. Since $E(G)\neq \emptyset$, there are two vertices $r,s\in V(G)$ with $(r,s)\in E(G)$. Assume that $v'$ is a vertex of $G_{I_{G\rightarrow H}^k}^{(u_{\psi_1}^k,u_{\psi_1}^{k-1}u_{\psi_2})}$ which belongs to the same connected component as $u_{\psi_1}^k$. Since $(p,q), (q,p)\notin E(H)$, we conclude that $v'$ is divisible by $x_{rp}^kx_{sp}^k$. As a consequence, $u_{\psi_1}^{k-1}u_{\psi_2}$ does not belong to the same connected component as $u_{\psi_1}^k$. In particular, $G_{I_{G\rightarrow H}^k}^{(u_{\psi_1}^k,u_{\psi_1}^{k-1}u_{\psi_2})}$ is not connected. Thus Theorem \ref{ref:BHN} implies that $I_{G\rightarrow H}^k$ is not linearly related, a contradiction. Therefore $H$ is a complete graph.

To complete the proof of (v)$\Rightarrow$(vi), we only need to prove that if $E(G)\neq \emptyset$, then $H$ is acyclic. By contradiction, assume that $H$ has a directed cycle. Let $C'$ denote the shortest directed cycle of $H$. Without loss of generality, we may assume that there is an integer $t$ with $3\leq t\leq m$ such that $V(C')=[t]$ and$$E(C')=\{(1,2), (2,3), \ldots, (t-1,t), (t,1)\}.$$As $H$ has no directed cycle of length $t-1$, we have $(1,3)\notin E(H)$. Since $E(G)\neq \emptyset$, without loss of generality, we may assume that $(1,2)\in E(G)$. Let $A\subseteq V(G)$ be the set of vertices $j$ for which there is a directed path (possibly of length zero) from $2$ to $j$ in $G$ (in particular, $2\in A$). Since $G$ is an acyclic graph, one has $1\notin A$.  Let $\xi_1 : G \to H$ be the weak homomorphism defined by $\xi_1(j)=1$, for each $j\in [n]$. Moreover, let $\xi_2 : G \to H$ be the weak homomorphism defined by $\xi_2(j)=3$ if $j\in A$ and $\xi_2(j)=2$, otherwise. Note that $u_{\xi_1}$ and $u_{\xi_2}$ belong to the set of minimal monomial generators of $I_{G\rightarrow H}$. Thus, $u_{\xi_1}^k$ and $u_{\xi_1}^{k-1}u_{\xi_2}$ belong to the set of minimal monomial generators of $I_{G\rightarrow H}^k$. Consider the relation graph $G_{I_{G\rightarrow H}^k}^{(u_{\xi_1}^k,u_{\xi_1}^{k-1}u_{\xi_2})}$. Note that every vertex of this graph is divisible by $x_{11}^{k-1}\cdots x_{n1}^{k-1}$ and divides$$\frac{x_{11}^k\cdots x_{n1}^kx_{12}\cdots x_{n2}x_{13}\cdots x_{n3}}{x_{13}x_{22}}.$$Moreover, no vertex of $G_{I_{G\rightarrow H}^k}^{(u_{\xi_1}^k,u_{\xi_1}^{k-1}u_{\xi_2})}$ is divisible by $x_{13}x_{22}$. Assume that $v''$ is a vertex of $G_{I_{G\rightarrow H}^k}^{(u_{\xi_1}^k,u_{\xi_1}^{k-1}u_{\xi_2})}$ which belongs to the same connected component as $u_{\xi_1}^k$. Suppose that $v''$ is not divisible by $x_{11}^kx_{21}^k$. Let $Q: u_{\xi_1}^k=w_0, w_1, \ldots, w_h=v''$ be a path between $u_{\xi_1}$ and $v''$ in $G_{I_{G\rightarrow H}^k}^{(u_{\xi_1}^k,u_{\xi_1}^{k-1}u_{\xi_2})}$. Suppose that that $i\in \{1, \ldots, h\}$ is the smallest integer for which $w_i$ is not divisible by $x_{11}^kx_{21}^k$. Therefore, either $x_{11}^kx_{21}^{k-1}\mid w_i$ and $x_{21}^k\nmid w_i$, or $x_{11}^{k-1}x_{21}^k\mid w_i$ and $x_{11}^k\nmid w_i$. We assume the first case, as the argument in the second case is similar. Since $x_{21}^k\nmid w_i$, we must have $x_{23}\mid w_i$. But this is a contradiction, as $(1,2)\in E(G)$ and $(1,3)\notin E(H)$. This contradiction shows that $v''$ is divisible by $x_{11}^kx_{21}^k$. As a consequence, $u_{\xi_1}^{k-1}u_{\xi_2}$ does not belong to the same connected component as $u_{\xi_1}^k$. In particular, $G_{I_{G\rightarrow H}}^{(u_{\xi_1},u_{\xi_2})}$ is not connected, which by Theorem  Theorem \ref{ref:BHN} implies that $I_{G\rightarrow H}^k$ is not linearly related, a contradiction. Hence, $H$ is an acyclic complete graph.

Next, we prove that (vi) implies (i). If $G$ is totally disconnected, then any map $\phi: G\to H$ is a weak homomorphism. This implies that$$I_{G\rightarrow H}=(x_{11}, \ldots, x_{1m})(x_{21}, \ldots, x_{2m})\cdots (x_{n1}, \ldots, x_{nm}),$$ which is a transversal matroidal ideal. As a consequence, $I_{G\rightarrow H}^k$ is a polymatroidal ideal for each integer $k\geq 1$. Now, suppose that $G$ is an acyclic graph and $H$ is an acyclic complete graph. By Lemma \ref{compacy}, we may assume that $H=D_m$. Let $G^{\ast}$ denote the transitive closure of $G$. Since $G$ is acyclic, $G^{\ast}$ is a simple directed graph. Furthermore, it follows from the structure of $D_m$  that $I_{G\rightarrow D_m}=I_{G^{\ast}\rightarrow D_m}$. Moreover, there is a poset $P$ on $[n]$ such that $i<_P j$ if and only if $(i,j)\in E(G^{\ast})$. Thus, $I_{G\rightarrow H}=I_{G^{\ast}\rightarrow D_m}$ is the ideal of poset homomorphisms defined by $P$ and the chain on $[m]$. Hence, we deduce from \cite[Theorem 5.1]{kkm} that $I_{G\rightarrow D_m}^k$ is weakly polymatroidal for each integer $k\geq 1$.
\end{proof}

In the following proposition, we compute the height of $I_{G\to H}$.

\begin{Proposition} \label{htdir}
Let $G$ and $H$ be directed graphs. Then $\height I_{G\rightarrow H}=|V(H)|$.  
\end{Proposition}

\begin{proof}
Assume that $V(G)=[n]$, $V(H)=[m]$. Since$$I_{G\rightarrow H}\subseteq (x_{11}, x_{12}, \ldots, x_{1m}),$$we deduce that $\height I_{G\rightarrow H}\leq m$. On the other hand, since the complete intersection$$(x_{11}x_{21}\cdots, x_{n1}, x_{12}x_{22}\cdots, x_{n2}, \ldots, x_{1m}x_{2m}\cdots, x_{nm})$$is contained in $I_{G\rightarrow H}$, it follows that $\height I_{G\rightarrow H}\geq m$.
\end{proof}

Next, we study the projective dimension and the regularity of (symbolic) powers of ideals of weak graph homomorphisms of directed graphs. As the proofs are similar to the corresponding results for undirected graphs, we omit the proofs. 

A subgraph $H'$ of a directed graph $H$ is called an {\it induced subgraph} if for any two vertices $i,j\in V(H')$, one has $(i,j)\in E(H')$ if and only if $(i,j)\in E(H)$.

\begin{Lemma} \label{restr}
Let $G$ and $H$ be directed graphs and let $H'$ be an induced subgraph of $H$. Then for any $i,j\geq 0$ and for any integer $k\geq 1$, one has
\begin{itemize}
    \item [(i)] $\beta_{i,j}(I_{G\rightarrow H'}^k)\leq\beta_{i,j}(I_{G\rightarrow H}^k)$.
    \item [(ii)] $\beta_{i,j}(I_{G\rightarrow H'}^{(k)})\leq\beta_{i,j}(I_{G\rightarrow H}^{(k)})$.
\end{itemize}
\end{Lemma}

\begin{proof}
The proof is similar to that of Lemma \ref{restriction}.
\end{proof}

\begin{Theorem} \label{regdirect}
Let $G$ and $H$ be directed graphs and let $c$ denote the number of connected components of $G$. Then for any integer $k\geq 1$, one has
\begin{itemize}
    \item [(i)] $\projdim I_{G\rightarrow H}^k\geq |V(H)|-1$.
    \item [(ii)] $\projdim I_{G\rightarrow H}^{(k)}\geq |V(H)|-1$.
    \item [(iii)] $\reg I_{G\rightarrow H}^k\geq k|V(G)|+(|V(G)|-c)(\alpha(H)-1)$.
    \item [(iv)] $\reg I_{G\rightarrow H}^{(k)}\geq k|V(G)|+(|V(G)|-c)(\alpha(H)-1)$.
\end{itemize}
\end{Theorem}

\begin{proof}
The proof is similar to that of Theorem \ref{thm:regundir}.
\end{proof}

A natural question about the ideal of weak graph homomorphisms is to characterize all directed graphs $G$ and $H$ for which $I_{G\to H}$ is unmixed. However, answering this question seems difficult. So, we restrict ourselves to the case $H=D_m$. We first need the following lemma.

\begin{Lemma} \label{parti}
Let $\ell$ and $m$ be positive integers and let $A_1, \ldots, A_m$ be disjoint subsets of $[\ell]^m$ such that $A_1\cup\cdots\cup A_m=[\ell]^m$. Then there are $k\in [m]$ and vectors $\mathbf{a}_1, \ldots \mathbf{a}_{\ell}\in A_k$ such that the $k$th component of $\mathbf{a}_j$ is $j$, for each $j=1, \ldots, \ell$.   
\end{Lemma}

\begin{proof}
By contradiction, assume that for each $k\in [m]$, there is an integer $j_k\in [\ell]$ such that for any vectors $\mathbf{a}\in A_k$, the $k$th component of $\mathbf{a}$ is not equal to $j_k$. We show that for each $i=1, \ldots, m$, the inequality
\[
\begin{array}{rl}
|A_1\cup A_2\cup\cdots \cup A_i|\leq (\ell-1)\ell^{m-1}+(\ell-1)\ell^{m-2}+\cdots + (\ell-1)\ell^{m-i}
\end{array} \tag{2} \label{2}
\] 
holds. To prove the above inequality, for each $k=1, \ldots, m$, let $B_k$ denote the set of vectors $\mathbf{a}\in [\ell]^m$ such that the $k$th component of $\mathbf{a}$ is not equal to $j_k$. Therefore, $A_k\subseteq B_k$. We prove something stronger than inequality (\ref{2}). Indeed, using induction on $i$, we prove that for each $i=1, \ldots, m$, 
\[
\begin{array}{rl}
|B_1\cup B_2\cup\cdots \cup B_i|=(\ell-1)\ell^{m-1}+(\ell-1)\ell^{m-2}+\cdots + (\ell-1)\ell^{m-i}.
\end{array} \tag{3} \label{3}
\]  
Since $A_1\cup A_2\cup\cdots \cup A_i\subseteq B_1\cup B_2\cup\cdots\cup B_i$, the inequality (\ref{2}) follows from the above equality.

We first prove  equality (\ref{3}) for $i=1$. Indeed, by definition of $B_1$, for any vector $\mathbf{a}\in B_1$, there are $\ell-1$ possibilities for the first component of $\mathbf{a}$ and $\ell$ possibilities for the other components. Therefore,$$|B_1|=(\ell-1)\ell^{m-1}.$$Next, suppose that $i\geq 2$. Note that for any vector $\mathbf{a}$ in $B_i\setminus (B_1\cup\cdots\cup B_{i-1})$ and for each $k=1, \ldots, i-1$, the $k$th component of $\mathbf{a}$ is $j_k$, and the $i$th component of $\mathbf{a}$ is not $j_i$. Consequently,$$|B_i\setminus (B_1\cup\cdots\cup B_{i-1})|=(\ell-1)\ell^{m-i}.$$
Hence, the above equality and the induction hypothesis imply that
\begin{align*}
&|B_1\cup B_2\cup\cdots\cup B_i|=|B_1\cup B_2\cup\cdots\cup B_{i-1}|+|B_i\setminus (B_1\cup\cdots\cup B_{i-1})|\\ & = (\ell-1)\ell^{m-1}+(\ell-1)\ell^{m-2}+\cdots + (\ell-1)\ell^{m-i+1}+ (\ell-1)\ell^{m-i}.
\end{align*}

Finally, inequality (\ref{2}) for $i=m$ implies that
\begin{align*}
|A_1\cup A_2\cup\cdots \cup A_m|\leq (\ell-1)\ell^{m-1}+(\ell-1)\ell^{m-2}+\cdots + (\ell-1)=\ell^m-1.
\end{align*}
This contradicts $A_1\cup\cdots\cup A_m=[\ell]^m$.
\end{proof}

\begin{Definition} \label{defshrink}
Let $G$ be a directed graph on $[n]$ and let $C$ be a directed cycle of $G$. The graph obtained by shrinking $C$ is the graph $G'$ with vertex set $([n]\setminus V(C))\cup \{0\}$, (where $0$ is a new vertex) and edge set
\begin{align*}
 E(G') & =E(G\setminus V(C))\cup \big\{(0,i)\mid i\notin V(C) \ {\rm and} \ (j,i)\in E(G) \ {\rm for \ some} \ j\in V(C)\big\}\\ & \cup\big\{(i,0)\mid i\notin V(C) \ {\rm and} \ (i,j)\in E(G) \ {\rm for \ some} \ j\in V(C)\big\}.
\end{align*}
\end{Definition}

\begin{Example} \label{exshrink}
Let $G_1$ be the first graph in Figure \ref{fig:shrinking}. The graph obtained from $G_1$ by shrinking its triangle is graph $G_2$. Also, the graph obtained from $G_1$ by shrinking its $4$-cycle is $G_3$.

\begin{figure}[h]
    \centering
    \subfloat[$G_1$]{\includegraphics[scale=0.55]{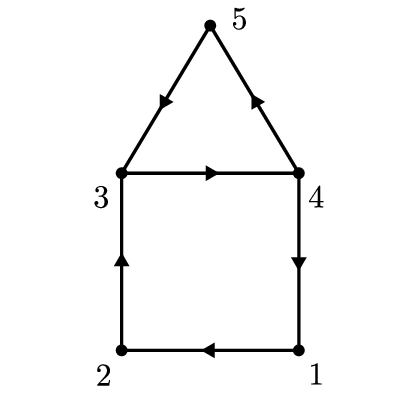}}
    \subfloat[$G_2$]{\includegraphics[scale=0.55]{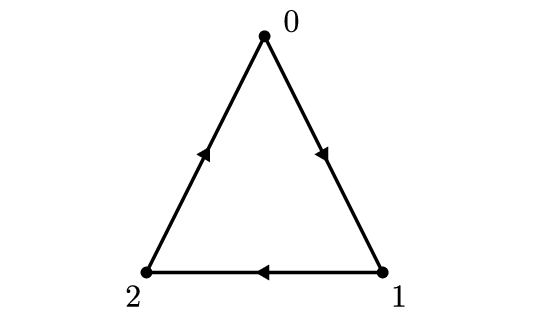}}
    \subfloat[$G_3$]{\includegraphics[scale=0.55]{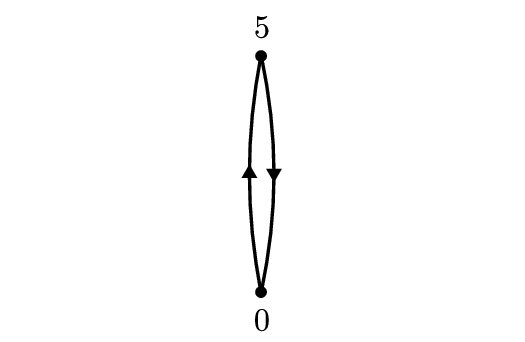}}
    \caption{$G_2$ and $G_3$ comes from $G_1$ by shrinking the $3$-cycle and the $4$-cycle, respectively.}
    \label{fig:shrinking}
\end{figure}

\end{Example}

As Example \ref{exshrink} shows even if $G$ is a simple graph, the graph $G'$ obtained by shrinking a directed cycle of $G$ is not necessarily simple. Indeed, for some vertex $i\in [n]\setminus V(C)$, both edges $(0,i)$ and $(i,0)$ may appear in $G'$. However, the definition of $I_{G\to H}$ makes sense when the underlying graph of $G$ has parallel edges. Moreover, after finitely many times shrinking the directed cycles, one obtains an acyclic directed graph.

\begin{Lemma} \label{shrinkunmix}
Let $G$ and $H$ be directed graphs such that $H$ is acyclic and $G$ has a directed cycle $C$. Let $G'$ denote the (not necssarily simple) graph obtained from $G$ by shrinking $C$. If $I_{G'\to H}$ is unmixed, then $I_{G\to H}$ is unmixed too.    
\end{Lemma}

\begin{proof}
Let $V(G)=[n]$ and $V(H)=[m]$. Without loss of generality, we may assume that $V(C)=[\ell]$, for some integer $\ell$ with $3\leq\ell\leq n$. and$$E(C)=\{(1,2), (2,3), \ldots, (\ell-1,\ell), (\ell,1)\}.$$Assume that $V(G')=\{0\}\cup\{\ell+1, \ldots, n\}$, where $0$ is the new vertex obtained from shrinking the cycle $C$. Set $S':=K[x_{ij} \mid i\in V(G'), j\in [m]]$. In particular, $I_{G'\rightarrow H}$ is an ideal of $S'$. Let$$I_{G'\rightarrow H}=\mathfrak{p}_1\cap \mathfrak{p}_2\cap\cdots \cap \mathfrak{p}_t$$be the  irredundant primary decomposition of $I_{G'\rightarrow H}$, where $\mathfrak{p}_1, \ldots, \mathfrak{p}_t$ are monomial prime ideals of $S'$. It follows from the assumption that $\height\mathfrak{p}_i=\height\mathfrak{p}_j$, for each pair of integers $1\leq i\leq j\leq t$. For each vector $\mathbf{a}=(a_1, \ldots, a_m)\in [\ell]^m$, let $f_{\mathbf{a}}: S'\longrightarrow S$ be the homomorphism defined by
\[
f_{\mathbf{a}}(x_{pq})=
\begin{cases}
  x_{a_qq} & \text{if $p=0$},\\
  x_{pq} & \text{if $p\neq 0$}.
\end{cases}
\]

We show that $I_{G\rightarrow H}\subseteq f_{\mathbf{a}}(I_{G'\rightarrow H})$. Indeed, fix a weak homomorphism $\phi\in \WHom (G,H)$. We know from Lemma \ref{dircyc} that $\phi(1)=\phi(2)=\cdots =\phi(\ell)$. We denote this common value with $b$. Therefore,$$u_{\phi}=x_{1b}\cdots x_{\ell b}x_{\ell+1\phi(\ell+1)}\cdots x_{n\phi(n)}.$$As a consequence, $x_{0b}x_{\ell+1\phi(\ell+1)}\cdots x_{n\phi(n)}\in I_{G'\rightarrow H}$. Note that$$f_{\mathbf{a}}(x_{0b}x_{\ell+1\phi(\ell+1)}\cdots x_{n\phi(n)})=x_{a_bb}x_{\ell+1\phi(\ell+1)}\cdots x_{n\phi(n)}$$divides $x_{1b}\cdots x_{\ell b}x_{\ell+1\phi(\ell+1)}\cdots x_{n\phi(n)}=u_{\phi}$. Thus, $u_{\phi}\in f_{\mathbf{a}}(I_{G'\rightarrow H})$ which implies that $I_{G\rightarrow H}\subseteq f_{\mathbf{a}}(I_{G'\rightarrow H})$. This yields that$$I_{G\rightarrow H}\subseteq \bigcap_{\mathbf{a}\in [\ell]^m}f_{\mathbf{a}}(I_{G'\rightarrow H}).$$We show that the reverse inclusion holds too. Choose an arbitrary monomial $v\in \bigcap_{\mathbf{a}\in [\ell]^m}f_{\mathbf{a}}(I_{G'\rightarrow H})$. Then for each vector $\mathbf{a}\in [\ell]^m$ there is a weak homomorphism $\psi_{\mathbf{a}}\in \WHom (G',H)$ such that $v$ is divisible by $\lcm_{\mathbf{a}\in[\ell]^m}\big(f_{\mathbf{a}}(u_{\psi_{\mathbf{a}}})\big)$. By Lemma \ref{parti} that are $k\in [m]$ and vectors $\mathbf{a}_1, \ldots, \mathbf{a}_{\ell}\in [\ell]^m$ such that $\psi_{\mathbf{a}_j}(0)=k$ and the $k$th component of $\mathbf{a}_j$ is $j$. Then $f_{\mathbf{a}_j}(u_{\psi_{\mathbf{a}_j}})$ is divisible by $x_{jk}$. As a consequence, $v$ is divisible by$$x_{1k}x_{2k}\cdots x_{\ell k}x_{\ell+1\psi_{\mathbf{a}_1}(\ell+1)}\cdots x_{n\psi_{\mathbf{a}_1}(n)}\in I_{G\rightarrow H},$$and we are done. 
Therefore,$$I_{G\rightarrow H}=\bigcap_{\mathbf{a}\in [\ell]^m}f_{\mathbf{a}}(I_{G'\rightarrow H}).$$Consequently,$$I_{G\rightarrow H}=\bigcap_{\mathbf{a}\in [\ell]^m}f_{\mathbf{a}}(I_{G'\rightarrow H})=\bigcap_{\mathbf{a}\in [\ell]^m}\bigcap_{i=1}^tf_{\mathbf{a}}(\mathfrak{p}_i)$$is a primary decomposition of $I_{G\rightarrow H}$. Since $\height f_{\mathbf{a}}(\mathfrak{p}_i)=\height \mathfrak{p}_i$, the assertion follows.  
\end{proof}

We are now ready to prove that $I_{G\to D_m}$ is an unmixed ideal.

\begin{Theorem} \label{unmixdir}
Let $G$ be a directed graph. Then $I_{G\to D_m}$ is unmixed. 
\end{Theorem}

\begin{proof}
If $G$ has no directed cycle, then $I_{G\rightarrow D_m}=I_{G^{\ast}\rightarrow D_m}$, where $G^{\ast}$ denotes the transitive closure of $G$. Moreover, there is a poset $P$ on $[n]$ such that $i<_P j$ if and only if $(i,j)\in E(G^{\ast})$. Thus, $I_{G\rightarrow D_m}$ is the ideal of poset homomorphisms defined by $P$ and and the chain on $[m]$. Therefore, it follows from \cite[Proposition 1.2]{fgh} that $I_{G\rightarrow D_m}$ is unmixed.

Now, suppose that $G$ has at least one directed cycle. 
By repeatedly shrinking the directed cycles of $G$, one obtains an acyclic graph $G''$. It follows from the preceding paragraph that $I_{G''\rightarrow D_m}$ is unmixed. Therefore, the Lemma \ref{shrinkunmix} implies that  $I_{G\rightarrow D_m}$ is unmixed.
\end{proof}

\begin{Definition}
Let $G$ be a directed graph. The simple acyclic directed graph obtained from $G$ by shrinking directed cycles and taking transitive closure is the following graph. First we repeat the procedure of shrinking the directed cycles until the resulting graph is a simple acyclic directed one. Then we take the transitive closure of the new graph.
\end{Definition}

In the following Theorem, we compute the projective dimension of $I_{G\to D_m}$.

\begin{Theorem} \label{pddirected}
Let $G$ be any directed graph. Then $\projdim I_{G\rightarrow D_m} = \alpha (G^*)(m-1)$, where $G^*$ is the simple acyclic directed graph obtained from $G$ by shrinking directed cycles and taking transitive closure.
\end{Theorem}

\begin{proof}
If $G$ has no directed cycle, then $I_{G\rightarrow D_m}=I_{G^{\ast}\rightarrow D_m}$, where $G^{\ast}$ denotes the transitive closure of $G$. Moreover, there is a poset $P$ on $[n]$ such that $i<_P j$ if and only if $(i,j)\in E(G^{\ast})$. Thus, $I_{G\rightarrow D_m}$ is the ideal of poset homomorphisms defined by $P$ and and the chain on $[m]$. Therefore, the assertion follows from \cite[Theorem 4.3]{kkm}.

Now, suppose that $G$ has a directed cycle, say $C$, as a subgraph. Without loss of generality, we may assume that $V(C)=[\ell]$, for some integer $\ell$ with $3\leq\ell\leq n$. and$$E(C)=\{(1,2), (2,3), \ldots, (\ell-1,\ell), (\ell,1)\}.$$It follows from Lemma \ref{dircyc} that for each weak homomorphism $\phi\in \WHom(G,D_n)$, one has $\phi(1)=\phi(2)=\cdots =\phi(\ell)$. Let $G'$ denote the (not necessarily simple) directed graph obtained from $G$ by shrinking the cycle $C$. Note that $I_{G\rightarrow D_m}$ and $I_{G'\rightarrow D_m}$ have the same LCM lattice. Hence, it follows from \cite[Theorem 2.1]{gpw} that $\projdim I_{G\rightarrow D_m}=\projdim I_{G'\rightarrow D_m}$. By repeatedly shrinking the directed cycles of $G$, one obtains an acyclic graph $G''$. Hence, it follows from the first paragraph of the proof that $\projdim I_{G\rightarrow D_m} = \alpha (G^*)(m-1)$.
\end{proof}

As a consequence of Theorem \ref{pddirected}, we are able to characterize all directed graphs $G$ for which $I_{G\to D_m}$ is Cohen-Macaulay.

\begin{Corollary} \label{CMdir}
Let $G$ be a directed graph on $n$ vertices. Then for an integer $m\geq 2$, the ideal $I_{G\rightarrow D_m}$ is Cohen-Macaulay if and only if up to a relabeling of vertices, $G^*=D_r$ for some $r$, where $G^*$ is obtained by shrinking directed cycles and taking transitive closure.
\end{Corollary}

\begin{proof}
We know that $I_{G\rightarrow D_m}$ is Cohen-Macaulay if and only if $\height I_{G\rightarrow D_m}=\projdim S/I_{G\rightarrow D_m}$. It follows from Proposition \ref{htdir} and Theorem \ref{pddirected} that $I_{G\rightarrow D_m}$ is Cohen-Macaulay if and only if $\alpha(G^*)=1$. Since $G^*$ has no directed cycle, it follows from Lemma \ref{compacy} that $\alpha(G^*)=1$ if and only if up to a relabeling of vertices, $G^*=D_r$, for some $r$.
\end{proof}

The following theorem is the last result of this paper and shows that if $I_{G\to H}$ is unmixed, then the depth function of its symbolic powers is nonincreasing.

\begin{Theorem} \label{depthdecr}
Let $G$ and $H$ be directed graphs such that $I:=I_{G\rightarrow H}$ is unmixed. Then$$\depth S/I\geq \depth S/I^{(2)}\geq \depth S/I^{(3)}\geq \cdots.$$In particular, for each directed graph $G$, $$\depth S/I_{G\rightarrow D_m}\geq \depth S/I_{G\rightarrow D_m}^{(2)}\geq \depth S/I_{G\rightarrow D_m}^{(3)}\geq \cdots.$$
\end{Theorem}

\begin{proof}
Fix an integer $k\geq 2$. We show that $\depth S/I^{(k)}\geq \depth S/I^{(k-1)}$. Assume that $V(G)=[n]$ and $V(H)=[m]$. Let$$I=\mathfrak{p}_1\cap\mathfrak{p}_2\cap\cdots\cap\mathfrak{p}_t$$be the irredundant primary decomposition of $I$, where $\mathfrak{p}_1, \ldots, \mathfrak{p}_t$ are monomial prime ideals of $S$. Since $I$ is unmixed, it follows from Proposition \ref{htdir} that $\height \mathfrak{p}_i=m$, for each $i=1, \ldots, t$. Moreover, as the complete intersection$$(x_{11}x_{21}\cdots x_{n1},x_{12}x_{22}\cdots x_{n2}, \ldots, x_{1m}x_{2m}\cdots x_{nm})$$is contained in $I$, it follows that for each $i=1, \ldots, t$ and for each $j=1,, \ldots, m$,$$|\mathfrak{p}_i\cap \{x_{1j},x_{2j}, \ldots, x_{nj}\}|\geq 1.$$As $\height\mathfrak{p}_i=m$, we deduce that equality holds in the above inequality. As a consequence,
\begin{align*}
 (I^{(k)}: x_{1j}x_{2j}\cdots x_{nj})=\Big(\bigcap_{i=1}^t\mathfrak{p}_i^k:x_{1j}x_{2j}\cdots x_{nj}\Big)=\bigcap_{i=1}^t\big(\mathfrak{p}_i^k:x_{1j}x_{2j}\cdots x_{nj})=\bigcap_{i=1}^t\mathfrak{p_i}^{k-1}=I^{(k-1)}.
\end{align*}
The assertion now follows from \cite[Corollary 1.3]{R}. The last part of the theorem is a consequence of Theorem \ref{unmixdir} and the first part.
\end{proof}

\begin{footnotesize}
{\bf Acknowledgments.} Computational experiments carried out using the computer algebra system Macaulay2~\cite{M2} have yielded several valuable insights. The third author is supported by a FAPA grant from Universidad de los Andes. The first and the second authors were supported by Scientific and Technological Research Council of Turkey T\"UB\.{I}TAK under the Grant No: 124F113. The first author is a member of GNSAGA Indam and he acknowledges their support.

{\bf Declaration of competing interest.}  The authors declare that they have no known competing financial interests or personal relationships that could have appeared to influence the work reported in this paper.\\

{\bf Data availability.} No data was used for the research described in the article.  
\end{footnotesize}

\end{document}